\documentclass{amsart}

\usepackage{latexsym}
\usepackage{amsmath}
\usepackage{amssymb}
\usepackage{amsthm}
\usepackage{amscd}
\usepackage{color}

\newtheorem{theorem}{Theorem}[section]
\newtheorem{lemma}[theorem]{Lemma}
\newtheorem{corollary}[theorem]{Corollary}
\newtheorem{proposition}[theorem]{Proposition}

\theoremstyle{definition}

\newtheorem{remark}[theorem]{Remark}
\newtheorem{definition}[theorem]{Definition}

\theoremstyle{remark}

\numberwithin{equation}{section}

\begin{document}

\title[Endpoint estimates for commutators]{Endpoint estimates for commutators with respect to the fractional integral operators on Orlicz-Morrey spaces}

\author{Naoya Hatano}

\address[Naoya Hatano]{Department of Department of Mathematics, Chuo University, 1-13-27, Kasuga, Bunkyo-ku, Tokyo 112-8551, Japan.}

\email[Naoya Hatano]{n.hatano.chuo@gmail.com}

\begin{abstract}
It is known that the necessary and sufficient conditions of the boundedness of commutators on Morrey spaces are given by Di Fazio, Ragusa and Shirai.
Moreover, according to the result of Cruz-Uribe and Fiorenza in 2003, it is given that the weak-type boundedness of the commutators of the fractional integral operators on the Orlicz spaces as the endpoint estimates.
In this paper, we gave the extention to the weak-type boundedness on the Orlicz-Morrey spaces.
\end{abstract}


\maketitle

\section{Introduction}

Throughout in this paper, denote by $L^0({\mathbb R}^n)$ the set of all measurable functions on ${\mathbb R}^n$.
Let $0<\alpha<n$.
The fractional integral operator $I_\alpha$ is defined by
\[
I_\alpha f(x)
\equiv
\int_{{\mathbb R}^n}\frac{f(y)}{|x-y|^{n-\alpha}}\,{\rm d}y
\]
for $f\in L^0({\mathbb R}^n)$.
Since it is well known that $I_\alpha=C(-\Delta)^{-\alpha/2}$ for some constant $C>0$, the operator $I_\alpha$ is used in many fields of mathematics such as partial differential equations, potential theory, and so on.

In this paper, we investigate the commutator with respect to the fractional integral operator $[b,I_\alpha]\equiv bI_\alpha-I_\alpha(b\cdot)$ for $b\in L^0({\mathbb R}^n)$.
In the investigation of the commutators, we sometimes assume that $b\in L_{\rm loc}^1({\mathbb R}^n)$ belongs to the John-Nireberg space ${\rm BMO}({\mathbb R}^n)$.
The John-Nireberg space ${\rm BMO}({\mathbb R}^n)$ is defined by the space of all $b\in L_{\rm loc}^1({\mathbb R}^n)$ with the finite functional
\[
\|b\|_{\rm BMO}
\equiv
\sup_{Q\in{\mathcal Q}}
\frac1{|Q|}\int_Q|b(x)-b_Q|\,{\rm d}x,
\]
where for each $Q\in{\mathcal Q}({\mathbb R}^n)$,
\[
b_Q
\equiv
\frac1{|Q|}\int_Qb(x)\,{\rm d}x
\]
for $b\in L_{\rm loc}^1({\mathbb R}^n)$.

In 1982, Chanillo \cite{Chanillo82} proved the boundedness of $[b,I_\alpha]$ on the Lebesgue spaces, that is,
\[
\|[b,I_\alpha]f\|_{L^s}
\le C
\|b\|_{\rm BMO}\|f\|_{L^p},
\]
where $C>0$ is some positive constant independent of $b$ and $f$, and the condition of the exponents $p$ and $s$ is $1<p<s<\infty$ and $1/s=1/p-\alpha/n$.
Here, the necessary and sufficiently condition of the boundedness of $[b,I_\alpha]$ on the Lebesgue spaces is $b\in{\rm BMO}({\mathbb R}^n)$.
This result was extended to Morrey spaces by Di Fazio and Ragusa \cite{FaRa91} in 1991, and Shirai \cite{Shirai06-1} in 2006.
Afterwards, many authors investigated for some generalizations \cite{ArNa18,Nogayama19,Shirai06-2}.
On the other hand, as the end point estimates, Cruz-Uribe and Fiorenza \cite{CrFi03} gave the following modular weak-type inequality in 2003:
\[
|\{x\in{\mathbb R}^n\,:\,|[b,I_\alpha]f(x)|>1\}|
\le C
\Psi\left(
\int_{{\mathbb R}^n}\Phi(\|b\|_{\rm BMO}|f(x)|)\,{\rm d}x
\right),
\]
where we put
\[
\Phi(t)\equiv t\log(3+t), \quad
\Psi(t)\equiv\left(t\log\left(3+t^{\frac\alpha n}\right)\right)^{\frac n{n-\alpha}},
\]
and $C>0$ is some positive constant independent of $b$ and $f$.
Moreover, it is known that weak-$(1,n/(n-\alpha))$ boundedness of $[b,I_\alpha]$ does not hold (see \cite[p.\ 104]{CrFi03}).
Replacing $f\mapsto f/(\lambda t)$ for any $\lambda,t>0$, and using the subadditivity of $\Phi$ and $\Psi$, we have the norm estimate
\[
\|[b,I_\alpha]f\|_{{\rm W}L^{\Psi_1}}
\le C
\|f\|_{L^\Phi}
\]
as the weak-type boundedness of $[b,I_\alpha]$ for some Young function $\Psi_1$ and a constant $C>0$, easily.
Here, the Young function $\Psi_1$ is given in Lemma \ref{lem:elemental} to follow.

Here, we extended the end point estimates to the weak-type boundedness on Orlicz-Morrey spaces.
The set of all cubes whose edges are parallel to the coordinate axes is denoted by ${\mathcal Q}({\mathbb R}^n)$.
At first, we introduce the definitions to state the main theorem.

\begin{definition}
Let $0<q\le p<\infty$.
The Morrey space ${\mathcal M}^p_q({\mathbb R}^n)$ and its weak-type space ${\rm W}{\mathcal M}^p_q({\mathbb R}^n)$ are defined by the space of all $f\in L^0({\mathbb R}^n)$ with the finite norm
\[
\|f\|_{{\mathcal M}^p_q}
\equiv
\sup_{Q\in{\mathcal Q}}|Q|^{\frac1p-\frac1q}
\left(\int_Q|f(x)|^q\,{\rm d}x\right)^{\frac1q}
\]
and
\[
\|f\|_{{\rm W}{\mathcal M}^p_q}
\equiv
\sup_{t>0}
\left\|t\chi_{\{x\in{\mathbb R}^n\,:\,|f(x)|>t\}}\right\|_{{\mathcal M}^p_q},
\]
respectively.
\end{definition}

To mention the main theorem, we use the Orlicz-Morrey spaces of the second kind introduced by \cite{SST12} (see also \cite{HaSa21}).

\begin{definition}\label{def:Mp-LlogL}
Let $1<p<\infty$.
\begin{itemize}
\item[{\rm (1)}] For each $Q\in{\mathcal Q}({\mathbb R}^n)$, define
\[
\|f\|_{L\log L,Q}
\equiv
\inf\left\{
\lambda>0
\,:\,
\frac1{|Q|}\int_Q\frac{|f(x)|}\lambda\log\left(3+\frac{|f(x)|}\lambda\right)\,{\rm d}x
\le1
\right\}.
\]

\item[{\rm (2)}] The Orlicz-Morrey space ${\mathcal M}^p_{L\log L}({\mathbb R}^n)$ is defined by the space of all $f\in L^0({\mathbb R}^n)$ with the finite norm
\[
\|f\|_{{\mathcal M}^p_{L\log L}}
\equiv
\sup_{Q\in{\mathcal Q}}|Q|^{\frac1p}
\|f\|_{L\log L,Q}.
\]
\end{itemize}
\end{definition}

We obtain the main theorem as follows.

\begin{theorem}\label{thm:main}
Let $0<\alpha<n$, $1<p<\infty$ and $1<t<s<\infty$.
Assume that
\[
\frac1s=\frac1p-\frac\alpha n,
\quad
\frac1p=\frac ts.
\]
Then $b\in{\rm BMO}({\mathbb R}^n)$ if and only if for all $f\in{\mathcal M}^p_{L\log L}({\mathbb R}^n)$, $[b,I_\alpha]f$ is well defined, and there exists $C>0$ independent of $b$ and $f$ such that
\[
\|[b,I_\alpha]f\|_{{\rm W}{\mathcal M}^s_t}
\le C
\|f\|_{{\mathcal M}^p_{L\log L}}
\]
hold.
\end{theorem}

There are some related results using the Orlicz-Morrey spaces as follows:
Shi, Arai and Nakai gave the boundedness of commutators on Orlicz-Morrey spaces of the first kind \cite{SAN21}.
Moreover, boundedness of its weak-type spaces is given by Kawasumi in \cite{Kawasumi23}.
Deringoz, Guliyev, Samko and Hasanov established the boundedness from Orlicz-Morrey spaces of the third kind to other Orlicz-Morrey spaces of the third kind \cite{DGH16,DGS15}.
Gogatishvili, Mustafayev and Ag\v cayazi considered the maximal commutator which corresponds to the case where $\alpha=0$ \cite{GMA18}.
Hakim and Sawano investigated the boundedness of $I_\alpha$ from the Calder\'{o}n-Lozanovski\v{\i} product to ${\mathcal M}^s_t({\mathbb R}^n)$ \cite{HaSa21}.

We will use standard notation for inequalities.
We use $C$ to denote a positive constant that may vary from one occurrence to another.
If $f\le Cg$, then we write $f\lesssim g$ or $g\gtrsim f$, and if $f\lesssim g\lesssim f$, then we write $f\sim g$.

We organize the remaining part of the paper as follows:
In Section \ref{s:statements}, we provide some known results for each function spaces.
In Section, \ref{ss:sharp}, we introduce the sparse theory on weak Morrey spaces to give some pointwise estimate for the sharp maximal functions of the commutators.
In Section \ref{s:Orlicz-fractional}, we provide some endpoint estimate for the Orlicz fractional maximal operators on the Orlicz-Morrey spaces with respect to general Young functions.
In Section \ref{s:modified}, we show the endpoint estimate for the modified commutators with positive kernel on Orlicz-Morrey spaces.
In Section \ref{s:proof}, we give the proof of the main theorem.

\section{Known statements for each function spaces}\label{s:statements}

In this section, we provide some known results for each function spaces to prove the main theorem.

\subsection{weak Lebesgue and Morrey spaces}

\begin{definition}
Let $0<p<\infty$.
The weak Lebesgue space ${\rm W}L^p({\mathbb R}^n)$ is defined by the space of all $f\in L^0({\mathbb R}^n)$ with the finite quasi-norm
\[
\|f\|_{{\rm W}L^p}
\equiv
\sup_{t>0}t|\{x\in{\mathbb R}^n\,:\,|f(x)|>t\}|^{\frac1p}.
\]
\end{definition}

It is well known that the triangle inequality of $\|\cdot\|_{{\rm W}L^p}$ holds as follows:
\[
\|f+g\|_{{\rm W}L^p}
\le
\max\left(2,2^{\frac1p}\right)
\left(\|f\|_{{\rm W}L^p}+\|g\|_{{\rm W}L^p}\right).
\]
Moreover, when $p>1$, $\|\cdot\|_{{\rm W}L^p}$ is normable.

\begin{proposition}[{\cite[Exercise 1.1.12]{Grafakos14}}]
\label{prop:normable-wLp}
Let $0<p_0<p<\infty$.
Then
\[
\|f\|_{{\rm W}L^p}
\sim
\sup_{0<|E|<\infty}|E|^{\frac1p-\frac1{p_0}}\|f\chi_E\|_{L^{p_0}}
\]
for $f\in{\rm W}L^p({\mathbb R}^n)$, where the supremum is taken over all measurable sets $E$ with $0<|E|<\infty$.
\end{proposition}

Additoinally, we have the boundedness and the Fefferman-Stein type vector-valued ineqality for the Hardy-Littlewood maximal operaor on the weak Lebesgue spaces is given to follow.

\begin{definition}
For $f\in L_{\rm loc}^1({\mathbb R}^n)$, define
\[
Mf(x)
\equiv
\sup_{Q\in{\mathcal Q}}\frac{\chi_Q(x)}{|Q|}\int_Q|f(y)|\,{\rm d}y,
\quad x\in{\mathbb R}^n.
\]
\end{definition}

\begin{proposition}[{\cite[Theorem 3.7]{CGMP06} and \cite[Theorem 6.1]{HKS22}}]
\label{prop:FS-WLp}
Let $1<p<\infty$ and $1<q\le\infty$.
Then the following assertions hold{\rm :}
\begin{itemize}
\item[{\rm (1)}] For all $f\in{\rm W}L^p({\mathbb R}^n)$,
\[
\|Mf\|_{{\rm W}L^p}
\lesssim
\|f\|_{{\rm W}L^p}.
\]
\item[{\rm (2)}]
For all $\{f_j\}_{j=1}^\infty\subset{\rm W}L^p({\mathbb R}^n)$,
\[
\left\|\left(\sum_{j=1}^\infty(Mf)^q\right)^{\frac1q}\right\|_{{\rm W}L^p}
\lesssim
\left\|\left(\sum_{j=1}^\infty|f|^q\right)^{\frac1q}\right\|_{{\rm W}L^p},
\]
where the case $q=\infty$ requires a natural modification.
\end{itemize}
\end{proposition}

As Morrey spaces extend the Lebesgue spaces, weak Morrey spaces are also extensions to the weak Lebesgue spaces: ${\rm W}{\mathcal M}^p_p({\mathbb R}^n)={\rm W}L^p({\mathbb R}^n)$.
Moreover, it is easy to see that the weak Morrey quasi-norm $\|\cdot\|_{{\rm W}{\mathcal M}^p_q}$, $0<q\le p<\infty$, has another expression:
\[
\|f\|_{{\rm W}{\mathcal M}^p_q}
=
\sup_{Q\in{\mathcal Q}}
|Q|^{\frac1p-\frac1q}\|f\chi_Q\|_{{\rm W}L^q}.
\]
The quasi-norm $\|\cdot\|_{{\rm W}{\mathcal M}^p_q}$ of the indicator functions over cubes can be calculated by
\[
\|\chi_Q\|_{{\mathcal M}^p_q}
=
|Q|^{\frac1p},
\]
clearly, as follows:

\begin{proposition}\label{thm:chiQwMpq}
Let $0<q\le p<\infty$, and let $Q\in{\mathcal Q}({\mathbb R}^n)$.
Then
\[
\|\chi_Q\|_{{\rm W}{\mathcal M}^p_q}
=
|Q|^{\frac1p}.
\]
\end{proposition}

It is known that $M$ is a bounded operator on ${\rm W}{\mathcal M}^p_q({\mathbb R}^n)$, too.

\begin{proposition}[{\cite[Theorem 3.1]{KaNa21}}]\label{prop:wMpq-bdd}
Let $1<q\le p<\infty$.
Then
\[
\|Mf\|_{{\rm W}{\mathcal M}^p_q}
\lesssim
\|f\|_{{\rm W}{\mathcal M}^p_q}
\]
for all $f\in{\rm W}{\mathcal M}^p_q({\mathbb R}^n)$.
\end{proposition}

\subsection{Orlicz and weak Orlicz spaces}

In this section, we recall the definition of Orlicz, weak Orlicz and Orlicz-Morrey spaces, and give some boundedness properties of the Orlicz fractional maximal operators.

\begin{definition}[Young function]
A function $\Phi:[0,\infty)\to[0,\infty)$ is called a Young function, if it satisfies the following properties:
\begin{itemize}
\item[(1)] $\Phi(t)>0$ for all $t>0$.
\item[(2)] $\lim\limits_{t\to0^+}\Phi(t)=0$.
\item[(3)] $\Phi$ is convex.
\end{itemize}
\end{definition}

We remark that any Young function is continuous and bijective on $[0,\infty)$.
So, we can consider the inverse function of Young functions, simply.

\begin{definition}
Let $\Phi$ be a Young function.
\begin{itemize}
\item[{\rm (1)}] The Orlicz space $L^\Phi({\mathbb R}^n)$ is defined to be the space of all $f\in L^0({\mathbb R}^n)$ with the finite quasi-norm
\[
\|f\|_{L^\Phi}
\equiv
\inf\left\{
\lambda>0
\,:\,
\int_{{\mathbb R}^n}\Phi\left(\frac{|f(x)|}\lambda\right)\,{\rm d}x
\le1
\right\}.
\]

\item[{\rm (2)}] The weak Orlicz space ${\rm W}L^\Phi({\mathbb R}^n)$ is defined by the space of all $f\in L^0({\mathbb R}^n)$ with the finite quasi-norm
\[
\|f\|_{{\rm W}L^\Phi}
\equiv
\sup_{t>0}\left\|t\chi_{\{x\in{\mathbb R}^n\,:\,|f(x)|>t\}}\right\|_{L^\Phi}.
\]
\end{itemize}
\end{definition}

\begin{remark}
The author give a representation for the quasi-norm of the weak Orlicz spaces by the modular with Kawasumi and Ono in \cite[Lemma 2.4]{HKO23} as follows:
\[
\|f\|_{{\rm W}L^\Phi}
=
\inf\left\{
\lambda>0
\,:\,
\sup_{t>0}\Phi(t)\left|\left\{x\in{\mathbb R}^n\,:\,\frac{|f(x)|}\lambda>t\right\}\right|
\le1
\right\}.
\]
\end{remark}

\begin{lemma}[{cf. \cite[Example 45]{SDH20}}]\label{lem:Orlicz}
Let $\Phi$ be a Young funciton, and let $E$ be a measurable set.
Then
\[
\|\chi_E\|_{{\rm W}L^\Phi}
=
\left\{\Phi^{-1}\left(\frac1{|E|}\right)\right\}^{-1}.
\]
\end{lemma}

The boundedness of the Hardy-Littlewood maximal operator $M$ on ${\rm W}L^\Phi({\mathbb R}^n)$ is obtained by Kawasumi, Nakai and Shi, as follows.

\begin{definition}
A Young function $\Phi:[0,\infty)\to[0,\infty)$ is said to satisfy the $\nabla_2$-condition, denoted by $\Phi\in\nabla_2$, 
if there exists a constant $k>1$, called the $\nabla_2$-constant, such that
\begin{equation*}
\Phi(r)\le\frac1{2k}\Phi(kr)
\quad \text{for} \quad
r>0.
\end{equation*}
\end{definition}

\begin{theorem}[{\cite[Theorem 3.2]{KNS21}}]\label{thm:M-wLphi-bdd}
Let $\Phi$ be a Young function.
If $\Phi\in\nabla_2$, then there exists $C_\Phi>0$ such that
\[
\sup_{t>0}\Phi(t)\left|\left\{x\in{\mathbb R}^n\,:\,Mf(x)>t\right\}\right|
\le
\sup_{t>0}\Phi(t)\left|\left\{x\in{\mathbb R}^n\,:\,C_\Phi|f(x)|>t\right\}\right|.
\]
\end{theorem}

\section{Estimates of sharp maximal functions}\label{ss:sharp}

In this paper, we follow the discussion by Nakamura and Sawano in \cite{NaSa17}.
The paper \cite{Nogayama19} is also considered the similar argument.

\subsection{Norm estimates for the sharp maximal functions}

\begin{definition}\label{def:191228-1}
Let $f\in L^0({\mathbb R}^n)$ and $Q\in{\mathcal Q}({\mathbb R}^n)$.
\begin{itemize}
\item[(1)] The decreasing rearrangement of $f$ on ${\mathbb R}^n$ is defined by
\begin{equation*}
f^\ast(t)
\equiv
\inf\{\alpha>0:|\{x\in{\mathbb R}^n\,:\,|f(x)|>\alpha\}|\le t\},
\quad
0<t<\infty.
\end{equation*}
\item[(2)] The {\it local mean oscillation} of $f$ on $Q$ is defined by
\begin{equation*}
\omega_\lambda(f;Q)
\equiv
\inf_{c\in{\mathbb C}}((f-c)\chi_Q)^\ast(\lambda|Q|),
\quad
0<\lambda<2^{-1}.
\end{equation*}
\item[(3)] Assume that the function $f$ is real--valued.
Then, the {\it median} of $f$ over $Q$, which is denoted by $m_f(Q)$, is a real number satisfying
\begin{equation*}
|\{x\in Q:f(x)<m_f(Q)\}|,
\quad
|\{x\in Q:f(x)>m_f(Q)\}|
\le\frac{|Q|}2.
\end{equation*}
\end{itemize}
\end{definition}

Note that the median $m_f(Q)$ is possibly non-unique.

Fix $Q_0\in{\mathcal Q}({\mathbb R}^n)$.
Denote by ${\mathcal D}(Q_0)$ the set of all cubes obtained by bisecting $Q_0$ finitely many times.
For $0<\lambda<2^{-1}$ and $Q_0\in{\mathcal Q}$, the {\it dyadic local sharp maximal operator} $M^{\sharp,d}_{\lambda,Q_0}$ is denoted by
\begin{equation*}
M^{\sharp,d}_{\lambda,Q_0}f(x)
\equiv
\sup_{Q\in{\mathcal D}(Q_0)}\omega_\lambda(f;Q)\chi_Q(x),
\quad
x\in{\mathbb R}^n,f\in L^0({\mathbb R}^n).
\end{equation*}
Moreover, we use the following sharp maximal operator
\begin{equation*}
M^{\sharp,d}_\lambda f(x)
\equiv
\sup_{Q_0\in{\mathcal Q}}\sup_{Q\in{\mathcal D}(Q_0)}\omega_\lambda(f;Q)\chi_Q(x),
\quad
x\in{\mathbb R}^n,f\in L^0({\mathbb R}^n).
\end{equation*}
Let $f\in L^1_{\rm loc}({\mathbb R}^n)$.
The {\it Fefferman-Stein sharp maximal function} is defined by
\begin{equation*}
f^{\sharp,\eta}(x)
\equiv
\sup_{Q\in{\mathcal Q}}
\left(\frac{\chi_Q(x)}{|Q|}\int_Q|f(y)-f_Q|^\eta\,{\rm d}y\right)^{\frac1\eta}.
\quad x\in{\mathbb R}^n,
\end{equation*}
When $\eta=1$, $f^{\sharp,1}$ abbreviates to $f^\sharp$.
Jawerth and Torchinsky proved a pointwise equivalence between these two types of the sharp maximal operators in \cite{JaTo85}:
\begin{equation}\label{eq:sharp-maximal}
M^{(\eta)}M^{\sharp,d}_\lambda f(x)
\sim
f^{\sharp,\eta}(x),
\quad
x\in{\mathbb R}^n
\end{equation}
for sufficiently small $\lambda>0$.

Here, some important norm estimates for the sharp maximal function $M^{\sharp,d}_\lambda f$ on weak Morrey spaces can be provided as follows.

\begin{theorem}\label{thm:191228-1}
Let $0<s\le q\le p<\infty$.
Then for all $f\in L^0({\mathbb R}^n)$,
\begin{equation*}
\|f\|_{{\rm W}{\mathcal M}^p_q}
\sim
\left\|M^{\sharp,d}_\lambda f\right\|_{{\rm W}{\mathcal M}^p_q}
+
\|f\|_{{\rm W}{\mathcal M}^p_s}.
\end{equation*}
\end{theorem}

\begin{theorem}\label{thm:230124-1}
Let $0<q\le p<\infty$.
If $f\in L^0({\mathbb R}^n)$ satisfies
\[
m_f(2^lQ)\to0
\]
as $l\to\infty$ for any $Q\in{\mathcal Q}({\mathbb R}^n)$ and for some medians $\{m_f(2^lQ)\}_{l\in{\mathbb N}_0}$, then
\[
\|f\|_{{\rm W}{\mathcal M}^p_q}
\lesssim
\left\|M_\lambda^{\sharp,d}f\right\|_{{\rm W}{\mathcal M}^p_q}.
\]
\end{theorem}

\begin{corollary}\label{cor:sharp-maximal}
Let $0<q\le p<\infty$.
If $f\in L^0({\mathbb R}^n)$ satisfies
\[
m_f(2^lQ)\to0
\]
as $l\to\infty$ for any $Q\in{\mathcal Q}({\mathbb R}^n)$ and for some medians $\{m_f(2^lQ)\}_{l\in{\mathbb N}_0}$, then
\[
\|f\|_{{\rm W}{\mathcal M}^p_q}
\sim
\left\|M_\lambda^{\sharp,d}f\right\|_{{\rm W}{\mathcal M}^p_q}.
\]
\end{corollary}

To prove Theorems \ref{thm:191228-1} and \ref{thm:230124-1}, we use the following pointwise estimate for the sparse family.
We say that the family $\{Q^k_j\}_{k\in{\mathbb N}_0,j\in J_k}\subset{\mathcal Q}({\mathbb R}^n)$ is a sparse family if the following properties hold:
for each $k\in{\mathbb N}_0$,
\begin{itemize}
\item[(1)] the cubes $\{Q^k_j\}_{j\in J_k}$ are disjoint;
\item[(2)] if $\Omega_k\equiv\bigcup_{j\in J_k}Q^k_j$, then $\Omega_{k+1}\subset\Omega_k$;
\item[(3)] $2|\Omega_{k+1}\cap Q^k_j|\le|Q^k_j|$ for all $j\in J_k$.
\end{itemize}

\begin{proposition}[{\cite{Lerner13}}]\label{prop:Lerner}
Let $f\in L^0({\mathbb R}^n)$, $Q_0\in{\mathcal Q}({\mathbb R}^n)$ and $\lambda_n\equiv2^{-n-2}$.
Then there exists a sparse family of $\{Q^k_j\}_{k\in{\mathbb N}_0,J_k}\subset{\mathcal D}(Q_0)$ such that
\[
|f(x)-m_f(Q_0)|
\le
4M^{\sharp,d}_{\lambda_n,Q_0}f(x)
+
2\sum_{k\in{\mathbb N}_0}\sum_{j\in J_k}\omega_{\lambda_n}(f;Q^k_j)\chi_{Q^k_j}(x),
\quad \text{a.e. $x\in Q_0$}.
\]
\end{proposition}

To obtain the condition
\[
m_f(2^lQ)\to0
\]
as $l\to\infty$ for any $Q\in{\mathcal Q}({\mathbb R}^n)$, we use the following lemma.

\begin{lemma}\label{lem:median0}
Let $f\in L^0({\mathbb R}^n)$.
If $Mf\in{\rm W}L^\Phi({\mathbb R}^n)$ for some Young function $\Phi:[0,\infty)\to[0,\infty)$, then
\[
\lim_{l\to\infty}m_f(2^lQ)=0
\]
for all $Q\in{\mathcal Q}({\mathbb R}^n)$.
\end{lemma}

\subsection{Proof of Theorem \ref{thm:191228-1}}

At first, to see the inequality \lq\lq\ $\lesssim$ '', we may show that for each $Q_0\in{\mathcal Q}({\mathbb R}^n)$,
\[
\|f\chi_{Q_0}\|_{{\rm W}L^q}
\lesssim
\left\|\left(M^{\sharp,d}_{\lambda_n,Q_0}f\right)\chi_{Q_0}\right\|_{{\rm W}L^q}
+
|Q_0|^{\frac1q-\frac1s}
\|f\chi_{Q_0}\|_{{\rm W}L^s}.
\]
By the triangle inequality for the quasi-norm $\|\cdot\|_{{\rm W}L^q}$,
\begin{equation}\label{eq:230124-1}
\|f\chi_{Q_0}\|_{{\rm W}L^q}
\lesssim
\|(f-m_f(Q_0))\chi_{Q_0}\|_{{\rm W}L^q}
+
|Q_0|^{\frac1q}|m_f(Q_0)|.
\end{equation}
For the first term, using Proposition \ref{prop:Lerner}, we have
\[
\|(f-m_f(Q_0))\chi_{Q_0}\|_{{\rm W}L^q}
\lesssim
\left\|\left(M^{\sharp,d}_{\lambda_n,Q_0}f\right)\chi_{Q_0}\right\|_{{\rm W}L^q}
+
{\bf I},
\]
where, we defined
\[
{\bf I}
\equiv
\left\|
\sum_{k\in{\mathbb N}_0}\sum_{j\in J_k}\omega_{\lambda_n}(f;Q^k_j)\chi_{Q^k_j}
\right\|_{{\rm W}L^q}.
\]
Since $\{Q^k_j\}_{k\in{\mathbb N}_0,j\in J_k}$ is a sparse family, we have the pointwise estimate $\chi_{Q^k_j}\lesssim M\chi_{\Omega_{k+1}^c\cap Q^k_j}$, and then for $\eta>\max(1,q^{-1})$, Proposition \ref{prop:FS-WLp} follows
\begin{align*}
{\bf I}
&\lesssim
\left\|
\sum_{k\in{\mathbb N}_0}\sum_{j\in J_k}\omega_{\lambda_n}(f;Q^k_j)
\left(M\chi_{\Omega_{k+1}^c\cap Q^k_j}\right)^\eta
\right\|_{{\rm W}L^q}
\lesssim
\left\|
\sum_{k\in{\mathbb N}_0}\sum_{j\in J_k}\omega_{\lambda_n}(f;Q^k_j)
\chi_{\Omega_{k+1}^c\cap Q^k_j}
\right\|_{{\rm W}L^q}\\
&\le
\left\|
\sum_{k\in{\mathbb N}_0}\sum_{j\in J_k}\left(M^{\sharp,d}_{\lambda_n,Q_0}f\right)
\chi_{\Omega_{k+1}^c\cap Q^k_j}
\right\|_{{\rm W}L^q}
\le
\left\|
\left(M^{\sharp,d}_{\lambda_n,Q_0}f\right)\chi_{Q_0}
\right\|_{{\rm W}L^q}.
\end{align*}

Next, we consider the second term in \eqref{eq:230124-1}.
For $\lambda\in(0,2^{-1})$ and $s_0\in(0,s)$, we see that
\begin{align*}
|m_f(Q_0)|
&\le
(f\chi_{Q_0})^\ast(\lambda|Q_0|)
\le
\left(
\frac1{\lambda|Q_0|}\int_0^{\lambda|Q_0|}(f\chi_{Q_0})^\ast(t)^{s_0}\,{\rm d}t
\right)^{\frac1{s_0}}\\
&\lesssim
\left(\frac1{|Q_0|}\int_{Q_0}|f(x)|^{s_0}\,{\rm d}x\right)^{\frac1{s_0}}
\lesssim
\frac1{|Q_0|^{\frac1s}}\|f\chi_{Q_0}\|_{{\rm W}L^s},
\end{align*}
where in the final inequality we used Proposition \ref{prop:normable-wLp}.

Secondly, to see the opposite inequation \lq\lq\ $\gtrsim$ '', we may use the known pointwise estimate
\[
M^{\sharp,d}_\lambda f(x)
\lesssim
M(|f|^\eta)(x)^{\frac1\eta}
\]
for any $\eta>0$ (see Lemma \ref{app:sharp} for detail).
Hence, taking $\eta\in(0,q)$, by Proposition \ref{prop:wMpq-bdd}, we obtain
\begin{align*}
\left\|M^{\sharp,d}_\lambda f\right\|_{{\rm W}{\mathcal M}^p_q}
\lesssim
\left\|(M(|f|^\eta))^{\frac1\eta}\right\|_{{\rm W}{\mathcal M}^p_q}
\lesssim
\|f\|_{{\rm W}{\mathcal M}^p_q},
\end{align*}
as desired.

\subsection{Proof of Theorem \ref{thm:230124-1}}

Fix $Q_0\in{\mathcal Q}({\mathbb R}^n)$.
By the quasi-triangle inequality $\|\cdot\|_{{\rm W}L^q}$,
\[
\|f\chi_{Q_0}\|_{{\rm W}L^q}
\lesssim
\|(f-m_f(2^lQ_0))\chi_{Q_0}\|_{{\rm W}L^q}
+
|Q_0|^{\frac1q}|m_f(2^lQ_0)|.
\]
By the assumption, it follows that
\[
\|f\chi_{Q_0}\|_{{\rm W}L^q}
\lesssim
\underset{l\to\infty}{\rm lim\,sup}\|(f-m_f(2^lQ_0))\chi_{Q_0}\|_{{\rm W}L^q},
\]
and hence, we focus on the quantity
$
\|(f-m_f(2^lQ_0))\chi_{Q_0}\|_{{\rm W}L^q}
$.
Using Proposition \ref{prop:Lerner}, we have
\[
\|(f-m_f(2^lQ_0))\chi_{Q_0}\|_{{\rm W}L^q}
\lesssim
\left\|\left(M^{\sharp,d}_{\lambda_n}f\right)\chi_{Q_0}\right\|_{{\rm W}L^q}
+
{\bf II},
\]
where, we defined
\[
{\bf II}
\equiv
\left\|\left(
\sum_{k\in{\mathbb N}_0}\sum_{j\in J_k}\omega_{\lambda_n}(f;Q^k_j)\chi_{Q^k_j}
\right)\chi_{Q_0}\right\|_{{\rm W}L^q}.
\]
Here, we remark that the family $\{Q^k_j\}_{k\in{\mathbb N}_0,j\in J_k}\subset{\mathcal D}(2^lQ_0)$ is a sparse family generated by $2^lQ_0$.
Setting $\{Q_{0\nu}\}_{\nu=1}^{2^n}$ by a family of the dyadic children of $Q_0$, that is, the family $\{Q_{0\nu}\}_{\nu=1}^{2^n}$ satisfies that for all $\nu=1,\ldots,2^n$,
\[
Q_{0\nu}\in{\mathcal D}(Q_0),
\quad
\ell(Q_{0\nu})=\frac{\ell(Q_0)}2,
\]
we can decompose
\begin{align*}
{\bf II}
&\lesssim
\sum_{\nu=1}^{2^n}
\left\|\left(
\sum_{k\in{\mathbb N}_0}\sum_{j\in J_k,Q^k_j\subset Q_{0\nu}}
\omega_{\lambda_n}(f;Q^k_j)\chi_{Q^k_j}
\right)\chi_{Q_0}\right\|_{{\rm W}L^q}\\
&\qquad+
\sum_{\nu=1}^{2^n}
\left\|\left(
\sum_{k\in{\mathbb N}_0}\sum_{j\in J_k,Q^k_j\supsetneq Q_{0\nu}}
\omega_{\lambda_n}(f;Q^k_j)\chi_{Q^k_j}
\right)\chi_{Q_0}\right\|_{{\rm W}L^q}
\equiv
{\bf II}_1+{\bf II}_2.
\end{align*}
Remark that $\{Q_{0\nu}\}_{\nu=1}^{2^n}\subset{\mathcal D}(2^lQ_0)$ for all $l\in{\mathbb N}$.

For ${\bf II}_1$, we may have the same discussion in the proof of Theorem \ref{thm:191228-1}, and therefore,
\[
{\bf II}_1
\lesssim
\left\|\left(M^{\sharp,d}_{\lambda_n}f\right)\chi_{Q_0}\right\|_{{\rm W}L^q}.
\]

Finally, we estimate ${\bf II}_2$.
A geometric observation for the sparse family allows us to rewrite the summation of ${\bf II}_2$ as follows:
\[
{\bf II}_2
\lesssim
\sum_{\nu=1}^{2^n}
\left\|\left(
\sum_{m=1}^{l+1}\omega_{\lambda_n}(f;Q_{0\nu}^{(m)})\chi_{Q_{0\nu}^{(m)}}
\right)\chi_{Q_{0\nu}}\right\|_{{\rm W}L^q},
\]
where $Q_{0\nu}^{(m)}$ denotes the dyadic $m$-th ancestor of $Q_{0\nu}$, that is, $Q_{0\nu}^{(m)}$ is a unique dyadic cube with respect to $2^lQ_0$ whose side length is $2^m\ell(Q)$ and containing $Q_{0\nu}$ for each $\nu=1,\ldots,2^n$.
Remarking the relation $Q_{0\nu}\subset Q_{0\nu}^{(m)}$, we see that
\begin{align*}
{\bf II}_2
&\lesssim
\sum_{\nu=1}^{2^n}\sum_{m=1}^{l+1}
\omega_{\lambda_n}(f;Q_{0\nu}^{(m)})|Q_{0\nu}|^{\frac1q}
\le
\sum_{\nu=1}^{2^n}|Q_{0\nu}|^{\frac1q}
\sum_{m=1}^{l+1}\frac1{|Q_{0\nu}^{(m)}|^{\frac1q}}
\left\|\left(M^{\sharp,d}_{\lambda_n}f\right)\chi_{Q_{0\nu}^{(m)}}\right\|_{{\rm W}L^q}\\
&\le
\sum_{\nu=1}^{2^n}|Q_{0\nu}|^{\frac1q}
\sum_{m=1}^{l+1}\frac1{|Q_{0\nu}^{(m)}|^{\frac1p}}
\left\|M^{\sharp,d}_{\lambda_n}f\right\|_{{\rm W}{\mathcal M}^p_q}
\lesssim
|Q_0|^{\frac1q-\frac1p}
\left\|M^{\sharp,d}_{\lambda_n}f\right\|_{{\rm W}{\mathcal M}^p_q}.
\end{align*}

Gathering these estimates of ${\bf II}_1$ and ${\bf II}_2$, we conclude the desired result.

\subsection{Proof of Lemma \ref{lem:median0}}

Fix $Q\in{\mathcal Q}({\mathbb R}^n)$.
By Lemma \ref{lem:Orlicz}, we estimate
\begin{align*}
|m_f(2^lQ)|
&\le
(f\chi_{2^lQ})^\ast\left(\frac{|2^lQ|}4\right)
\le
4\inf_{x\in2^lQ}Mf(x)
\le
4\frac{\|Mf\|_{{\rm W}L^\Phi}}{\|\chi_{2^lQ}\|_{{\rm W}L^\Phi}}\\
&=
4\Phi^{-1}\left(\frac1{|2^lQ|}\right)\|Mf\|_{{\rm W}L^\Phi}
\to0,
\end{align*}
as $l\to\infty$ (see lemma \ref{app:sharp} for detail of the first and second inequalities.
This is the desired result.

\section{Endpoint estimates for the Orlicz fractional maximal operators}\label{s:Orlicz-fractional}

\begin{definition}
Let $0\le\alpha<n$, and let $\Phi$ be a Young function.
\begin{itemize}
\item[{\rm (1)}] For each $Q\in{\mathcal Q}({\mathbb R}^n)$, define
\[
\|f\|_{\Phi,Q}
\equiv
\inf\left\{
\lambda>0
\,:\,
\frac1{|Q|}\int_Q\Phi\left(\frac{|f(x)|}\lambda\right)\,{\rm d}x
\le1
\right\}.
\]
\item[{\rm (2)}] The Orlicz fractional maximal operator $M_{\alpha,\Phi}$ is defined by
\[
M_{\alpha,\Phi}f(x)
\equiv
\sup_{Q\in{\mathcal Q}}\chi_Q(x)\ell(Q)^\alpha\|f\|_{\Phi,Q},
\quad x\in{\mathbb R}^n.
\]
\item[{\rm (3)}] When $\alpha=0$, $M_{0,\Phi}$ abbreviates to $M_\Phi$ and is called the Orlicz maximal operator.
\item[{\rm (4)}] In particular, if $\Phi(t)=t\log(3+t)$, $\|\cdot\|_{\Phi,Q}$, $M_\Phi$ and $M_{\alpha,\Phi}$ are also written as $\|\cdot\|_{L\log L,Q}$, $M_{L\log L}$ and $M_{\alpha,L\log L}$, respectively.
\item[{\rm (5)}] When $\Phi(t)=e^t-1$, the Orlicz-average $\|\cdot\|_{\Phi,Q}$ is also written as $\|\cdot\|_{\exp(L),Q}$.
\end{itemize}
\end{definition}

The endpoint estimate for the Orlicz fractional maximal operators on Orlicz spaces are given by Cruz-Uribe and Fiorenza, as follows.

\begin{theorem}[{\cite[Theorem 3.3]{CrFi03}}]\label{thm:CF-M}
Let $0\le\alpha<n$, and let $\Phi$ be a function.
Set
\[
h_\Phi(t)
\equiv
\sup_{s>0}\frac{\Phi(st)}{\Phi(s)},
\quad t>0.
\]
Assume that the mapping
\[
t\mapsto\frac{\Phi(t)}{t^{\frac n\alpha}}
\]
is decreasing whenever $0<\alpha<n$.
If one put
\[
\Psi(t)
\equiv
\frac t{h_\Phi\left(t^{\frac\alpha n}\right)},
\quad t>0,
\]
then for all $f\in L^0({\mathbb R}^n)$,
\[
\Psi(|\{x\in{\mathbb R}^n\,:\,M_{\alpha,\Phi}f(x)>t\}|)
\lesssim
\int_{{\mathbb R}^n}\Phi(|f(x)|)\,{\rm d}x.
\]
Especially, when $\alpha=0$, the weak-type boundedness of the Orlicz-maximal operator is obtained, as follows{\rm :}
\[
|\{x\in{\mathbb R}^n\,:\,M_\Phi f(x)>t\}|
\lesssim
\int_{{\mathbb R}^n}\Phi(|f(x)|)\,{\rm d}x,
\quad
f\in L^0({\mathbb R}^n).
\]
\end{theorem}

To prove the main theorem, we gave the end point estimate for the Orlicz fractional maximal operators on Orlicz-Morrey spaces with respect to general Young functions.

\begin{definition}[{\cite{SST12}}]
Let $1\le p<\infty$, and let $\Phi$ be a Young function.
The Orlicz-Morrey space ${\mathcal M}^p_\Phi({\mathbb R}^n)$ is defined to be the space of all $f\in L^0({\mathbb R}^n)$ with the finite norm
\[
\|f\|_{{\mathcal M}^p_\Phi}
\equiv
\sup_{Q\in{\mathcal Q}}|Q|^{\frac1p}\|f\|_{\Phi,Q}.
\]
In particular, when $\Phi(t)=t\log(3+t)$, the Orilcz-Morrey spaces ${\mathcal M}^p_\Phi({\mathbb R}^n)$ are written as ${\mathcal M}^p_{L\log L}({\mathbb R}^n)$ defined in Definition \ref{def:Mp-LlogL}.
\end{definition}

\begin{remark}
It is mentioned by Iida in \cite[Remark 2]{Iida21} that the necessary and sufficiently condition of ${\mathcal M}^p_\Phi({\mathbb R}^n)\ne\{0\}$ is the pointwise inequality
\[
\Phi(t)\lesssim t^p\;\text{for}\;t\ge1
\]
holds.
Therefore, ${\mathcal M}^p_{L\log L}({\mathbb R}^n)\ne\{0\}$ if and only if $p>1$.
\end{remark}

\begin{theorem}\label{thm:220312-H1}
Let $0\le\alpha<n$, $1\le p<\infty$ and $1\le t\le s<\infty$, and let $\Phi$ be a Young function.
If
\[
\frac1s=\frac1p-\frac\alpha n,
\quad
\frac1p=\frac ts,
\]
then
\[
\|M_{\alpha,\Phi}f\|_{{\rm W}{\mathcal M}^s_t}
\sim
\|f\|_{{\mathcal M}^p_\Phi}
\]
for all $f\in{\mathcal M}^p_\Phi({\mathbb R}^n)$.
\end{theorem}

The case $\alpha=0$ and $\Phi(t)=t$ is given by Sawano in \cite[Theorem 27]{Sawano18}.

\begin{proof}[Proof of Theorem {\rm \ref{thm:220312-H1}}]
First, we prove the estimate
$
\|M_\Phi f\|_{{\rm W}{\mathcal M}^p_1}
\lesssim
\|f\|_{{\mathcal M}^p_\Phi}
$
which is the case $\alpha=0$.
Fix $\lambda>0$ and $Q\in{\mathcal Q}({\mathbb R}^n)$.
We split
\[
f
=
f\chi_{2Q}+f\chi_{{\mathbb R}^n\setminus2Q}
\equiv
f_1+f_2.
\]
Using Theorem \ref{thm:CF-M}, we have
\[
|\{y\in Q\,:\,M_\Phi f_1(y)>\lambda\}|
\lesssim
\int_{2Q}\Phi\left(\frac{|f(z)|}\lambda\right)\,{\rm d}z.
\]
By the convexity of $\Phi$,
\begin{align*}
1
&\lesssim
\frac1{|Q|}
\int_{2Q}
\Phi\left(\frac{|f(z)|}{\lambda|\{y\in Q\,:\,M_\Phi f_1(y)>\lambda\}|/|Q|}\right)
\,{\rm d}z.
\end{align*}
This proves
\begin{equation}\label{eq:M1}
\frac{\lambda|\{y\in Q\,:\,M_\Phi f_1(y)>\lambda\}|}{|Q|}
\lesssim
\|f\|_{\Phi,2Q}.
\end{equation}
Additionally, by the simple geometric observation,
\begin{align*}
M_\Phi f_2(y)
&\le
\sup_{R\in{\mathcal Q},\,R\supset Q}
\inf\left\{
\lambda>0
\,:\,
\frac1{|R|}
\int_{R}
\Phi\left(\frac{|f(z)|}\lambda\right)
\,{\rm d}z
\le1
\right\}\\
&\le
|Q|^{-\frac 1p}\|f\|_{{\mathcal M}^p_\Phi}
\end{align*}
for all $y\in Q$.
Then
\begin{equation}\label{eq:M2}
\|M_\Phi f_2\|_{{\rm W}{\mathcal M}^p_1}
\lesssim
\|f\|_{{\mathcal M}^p_\Phi}.
\end{equation}
Hence, combining these estimates \eqref{eq:M1} and \eqref{eq:M2}, we have
\[
\|M_\Phi f\|_{{\rm W}{\mathcal M}^p_1}
\lesssim
\|f\|_{{\mathcal M}^p_\Phi}.
\]

Second, we estimate
$
\|M_{\alpha,\Phi}f\|_{{\rm W}{\mathcal M}^s_t}
\lesssim
\|f\|_{{\mathcal M}^p_\Phi}
$
for all $0<\alpha<n$.
By the definitions of $M_\Phi$ and $\|\cdot\|_{{\mathcal M}^p_\Phi}$,
\[
\ell(Q)^\alpha\|f\|_{\Phi,Q}
\le
\ell(Q)^\alpha M_\Phi f(x),
\quad
\ell(Q)^\alpha\|f\|_{\Phi,Q}
\le
\ell(Q)^{\alpha-\frac np}\|f\|_{{\mathcal M}^p_\Phi},
\]
respectively.
Then
\begin{align*}
M_{\Phi,\alpha}f(x)
&\le
\sup_{Q\in{\mathcal Q}}
\min\left(
\ell(Q)^\alpha M_\Phi f(x),\ell(Q)^{\alpha-\frac np}\|f\|_{{\mathcal M}^p_\Phi}
\right)\\
&\le
\|f\|_{{\mathcal M}^p_\Phi}^{\frac{\alpha p}n}M_\Phi f(x)^{1-\frac{\alpha p}n}.
\end{align*}
Therefore, by the case $\alpha=0$,
\begin{align*}
\|M_{\Phi,\alpha}f\|_{{\rm W}{\mathcal M}^s_t}
\le
\|f\|_{{\mathcal M}^p_\Phi}^{\frac{\alpha p}n}\|M_\Phi f\|_{{\rm W}{\mathcal M}^p_1}^{1-\frac{\alpha p}n}
\lesssim
\|f\|_{{\mathcal M}^p_\Phi}.
\end{align*}

Finally, we prove the opposite estimate
$
\|M_{\alpha,\Phi}f\|_{{\rm W}{\mathcal M}^s_t}
\ge
\|f\|_{{\mathcal M}^p_\Phi}
$.
For all $Q\in{\mathcal Q}({\mathbb R}^n)$, using Proposition \ref{thm:chiQwMpq}, we estimate
\begin{align*}
\|M_{\Phi,\alpha}f\|_{{\rm W}{\mathcal M}^s_t}
\ge
\left\|\chi_Q\cdot|Q|^{\frac\alpha n}\|f\|_{\Phi,Q}\right\|_{{\rm W}{\mathcal M}^s_t}
=
|Q|^{\frac1p}\|f\|_{\Phi,Q},
\end{align*}
as desired.
\end{proof}

\section{Modified commutators with positive kernel}\label{s:modified}

To mention the boundedness of $[b,I_\alpha]$ on the Orlicz-Morrey space ${\mathcal M}^p_{L\log L}({\mathbb R}^n)$, we rewrite the definition of $[b,I_\alpha]$ as follows.

\begin{definition}
Let $0<\alpha<n$.
When $b\in{\rm BMO}({\mathbb R}^n)$, define
\[
|b,I_\alpha|f(x)
\equiv
\int_{{\mathbb R}^n}\frac{|b(x)-b(y)|}{|x-y|^{n-\alpha}}f(y)\,{\rm d}y,
\quad x\in{\mathbb R}^n
\]
for $f\in L^0({\mathbb R}^n)$ as long as the integral converges.
\end{definition}

The commutators of positive kernels is first considered by Bramanti in \cite{Bramanti94}.
Based on the ideas, we may discuss the well-definedness of the commutator $|b,I_\alpha|$.

In this section, we prove the following theorem.

\begin{theorem}\label{thm:absolute}
Let $0<\alpha<n$, $1<t<s<\infty$ and $1<p<\infty$.
Assume that
\[
\frac1s=\frac1p-\frac\alpha n,
\quad
\frac1p=\frac ts.
\]
If $b\in{\rm BMO}({\mathbb R}^n)$, then
\[
\||b,I_\alpha|f\|_{{\rm W}{\mathcal M}^s_t}
\lesssim
\|f\|_{{\mathcal M}^p_{L\log L}}
\]
for all $f\in{\mathcal M}^p_{L\log L}({\mathbb R}^n)$.
\end{theorem}

Similar to the proof of Theorem 1.1 in \cite{CrFi03}, we obtain the endpoint modular inequality for $|b,I_\alpha|$, as follows.

\begin{theorem}\label{thm:C-U}
Let $0<\alpha<n$, and let $\Phi$ and $\Psi$ be Young functions defined by
\[
\Phi(t)\equiv t\log(3+t), \quad
\Psi(t)\equiv\left(t\log\left(3+t^{\frac\alpha n}\right)\right)^{\frac n{n-\alpha}}
\]
for all $t>0$.
If $b\in{\rm BMO}({\mathbb R}^n)$, then
\[
|\{x\in{\mathbb R}^n\,:\,||b,I_\alpha|f(x)|>1\}|
\lesssim
\Psi\left(
\int_{{\mathbb R}^n}\Phi(\|b\|_{\rm BMO}|f(x)|)\,{\rm d}x
\right)
\]
for all $f\in L\log L({\mathbb R}^n)$.
\end{theorem}

Here, we start the proof of the Theorem \ref{thm:absolute}.
Fix $f\in{\mathcal M}^p_{L\log L}({\mathbb R}^n)$ and $Q\in{\mathcal Q}({\mathbb R}^n)$.
We decompose
\[
f
=
f\chi_{2Q}+f\chi_{{\mathbb R}^n\setminus2Q}
=:
f_1+f_2.
\]

First, we estimate $\||b,I_\alpha|f_1\|_{{\rm W}{\mathcal M}^s_t}$.
To use Lemma \ref{lem:median0}, we comfirm $|b,I_\alpha|f_1\in{\rm W}L^{\Phi_0}({\mathbb R}^n)$ for some Young function $\Phi_0$.
By the subadditivity of $\Phi$ and $\Psi$, the end point estimate for $f_1$ in Theorem \ref{thm:C-U} is rewritten by
\[
|\{x\in{\mathbb R}^n\,:\,||b,I_\alpha|f_1(x)|>t\}|
\lesssim
\Psi\circ\Phi\left(\frac1t\right)\cdot
\Psi\left(
\int_{{\mathbb R}^n}\Phi(\|b\|_{\rm BMO}|f_1(x)|)\,{\rm d}x
\right)
\]
for all $t>0$.
Here, we can use the following elemental lemma for the function $\Psi\circ\Phi(t)$ (see Appendix \ref{App:elemental} for the proof of this lemma).
 
\begin{lemma}\label{lem:elemental}
Let $0<\alpha<n$, and let $\Phi$ and $\Psi$ be Young functions defined in Theorem {\rm \ref{thm:C-U}}.
Addtionally, set
\[
\Psi_0(t)
\equiv
\left(\frac t{\log^2(3+1/t)}\right)^{\frac n{n-\alpha}},
\quad t>0.
\]
Then the following assertions hold.
\begin{itemize}
\item[{\rm (1)}] $
\Psi\circ\Phi(t)
\lesssim
\left(t\log^2(3+t)\right)^{n/(n-\alpha)}
$.
\item[{\rm (2)}] There exists a Young function $\Psi_1\in\nabla_2$ such that $\Psi_1(t)\lesssim\Psi_0(t)$.
\end{itemize}
\end{lemma}

Thus, we see that
\[
\left(\frac t{\log^2(3+1/t)}\right)^{\frac n{n-\alpha}}
|\{x\in{\mathbb R}^n\,:\,|[b,I_\alpha]f_1(x)|>t\}|
\lesssim
\Psi\left(
\int_{{\mathbb R}^n}\Phi(\|b\|_{\rm BMO}|f_1(x)|)\,{\rm d}x
\right).
\]
It follows that $M(|b,I_\alpha|f_1)\in{\rm W}L^{\Psi_1}({\mathbb R}^n)$ for a Young function $\Psi_1\in\nabla_2$ given in Lemma \ref{lem:elemental} by the ${\rm W}L^{\Psi_1}({\mathbb R}^n)$-boundedness of $M$ (see Theorem \ref{thm:M-wLphi-bdd}), and then we can use Lemma \ref{lem:median0}, and obtain
\[
\lim_{l\to\infty}m_{|b,I_\alpha|f_1}(2^lQ)=0
\]
for all $Q\in{\mathcal Q}({\mathbb R}^n)$.
Combining Corollary \ref{cor:sharp-maximal}, we have
\[
\||b,I_\alpha|f_1\|_{{\rm W}{\mathcal M}^s_t}
\sim
\left\|M_\lambda^{\sharp,d}(|b,I_\alpha|f_1)\right\|_{{\rm W}{\mathcal M}^s_t}.
\]
Additionally, we use the following theorem.

\begin{theorem}\label{thm:sharp-pointwise}
Let $0<\alpha<n$, $1<p<\infty$ and $\eta>1$.
If
\[
b\in{\rm BMO}({\mathbb R}^n),
\quad
\frac1p-\frac\alpha n>0,
\]
then
\[
(|b,I_\alpha|f)^\sharp(x)
\lesssim
\|b\|_{\rm BMO}
\left(
M((I_\alpha f)^\eta)(x)^{\frac1\eta}+M_{\alpha,L\log L}f(x)
\right),
\quad x\in{\mathbb R}^n
\]
for all nonnegative functions $f\in{\mathcal M}^p_{L\log L}({\mathbb R}^n)$.
\end{theorem}

Although this theorem can be obtain as the similar proof of Theorem 1.3 in \cite{CrFi03}, we give the proof of it in Appendix \ref{app:sharp-pointwise-proof} jast in case.

This theorem is given as the proof of Theorem 1.3 in \cite{CrFi03}.
Hence,
\begin{align*}
\||b,I_\alpha|f_1\|_{{\rm W}{\mathcal M}^s_t}
&\lesssim
\left\|M^{\sharp,d}_\lambda(|b,I_\alpha|f_1)\right\|_{{\rm W}{\mathcal M}^s_t}
\lesssim
\|b\|_{\rm BMO}
\left(
\|I_\alpha(|f_1|)\|_{{\rm W}{\mathcal M}^s_t}
+
\left\|M_{\alpha,L\log L}f_1\right\|_{{\rm W}{\mathcal M}^s_t}
\right)\\
&\lesssim
\|f\|_{{\mathcal M}^p_{L\log L}}.
\end{align*}

Second, we estimate $\||b,I_\alpha|f_2\|_{{\rm W}{\mathcal M}^s_t}$.
We may use the characterization for the ${\rm BMO}({\mathbb R}^n)$ functional, as follows.

\begin{lemma}[{\cite[Lemma 5.1]{CrFi03}}]\label{lem:bmo}
Let $1\le p<\infty$.
Then the following assrtions hold for $b\in L_{\rm loc}^1({\mathbb R}^n)$.
\begin{itemize}
\item[{\rm (1)}] $\displaystyle
\|b\|_{\rm BMO}
\sim
\sup_{Q\in{\mathcal Q}}
\left(\frac1{|Q|}\int_Q|b(x)-b_Q|^p\,{\rm d}x\right)^{\frac1p}
$.
\item[{\rm (2)}] $\displaystyle
\|b\|_{\rm BMO}
\sim
\sup_{Q\in{\mathcal Q}}
\|b-b_Q\|_{\exp(L),Q}
$.
\item[{\rm (3)}] For all $j\in{\mathbb N}$,
\[
\|b-b_Q\|_{\exp(L),2^jQ}
\lesssim
j\|b\|_{\rm BMO}.
\]
\end{itemize}
\end{lemma}

Fix $x\in Q$.
We decompose
\begin{align*}
||b,I_\alpha|f_2(x)|
&\le
\int_{{\mathbb R}^n\setminus2Q}\frac{|b(x)-b_Q|}{|x-y|^{n-\alpha}}|f(y)|\,{\rm d}y
+
\int_{{\mathbb R}^n\setminus2Q}\frac{|b(y)-b_Q|}{|x-y|^{n-\alpha}}|f(y)|\,{\rm d}y\\
&\equiv
F_1(x)+F_2(x).
\end{align*}
In $F_1(x)$,
\begin{align*}
F(x)
&\lesssim
|b(x)-b_Q|\sum_{j=1}^\infty
\frac1{\ell(2^{j+1}Q)^{n-\alpha}}\int_{2^{j+1}Q\setminus2^jQ}|f(y)|\,{\rm d}y\\
&\le
|b(x)-b_Q|\sum_{j=1}^\infty
\ell(2^{j+1}Q)^\alpha\|f\|_{{\mathcal M}^p_1}\cdot|2^{j+1}Q|^{-\frac1p}\\
&\sim
|b(x)-b_Q|\|f\|_{{\mathcal M}^p_1}\cdot|Q|^{-\frac1s}.
\end{align*}
It follows that
\begin{align}\label{eq:F1}
\begin{split}
|Q|^{\frac1s-\frac1t}\|F_1\chi_Q\|_{{\rm W}L^t}
&\lesssim
\frac{\|(b-b_Q)\chi_Q\|_{{\rm W}L^t}}{|Q|^{\frac1t}}\|f\|_{{\mathcal M}^p_1}
\lesssim
\frac{\|(b-b_Q)\chi_Q\|_{L^t}}{|Q|^{\frac1t}}\|f\|_{{\mathcal M}^p_1}\\
&\lesssim
\|b\|_{{\rm BMO}}\|f\|_{{\mathcal M}^p_1},
\end{split}
\end{align}
where in the final inequality we used Lemma \ref{lem:bmo}.
In $F_2(x)$, using the $L\log L$-$\exp(L)$ duality for the probability measure, we have
\begin{align*}
F_2(x)
&\lesssim
\sum_{j=1}^\infty\ell(2^{j+1}Q)^\alpha
\|b-b_Q\|_{\exp(L),2^{j+1}Q}\|f\|_{L\log L,2^{j+1}Q}\\
&\le
\sum_{j=1}^\infty\ell(2^{j+1}Q)^\alpha
(j+1)\|b\|_{\rm BMO}
\|f\|_{{\mathcal M}^p_{L\log L}}\cdot|2^{j+1}Q|^{-\frac1p}\\
&\sim
\|b\|_{\rm BMO}\|f\|_{{\mathcal M}^p_{L\log L}}\cdot|Q|^{-\frac1s},
\end{align*}
where in the second inequality we used Lemma \ref{lem:bmo}.
It follows that
\begin{equation}\label{eq:F2}
|Q|^{\frac1s-\frac1t}\|F_2\|_{{\rm W}L^t}
\lesssim
\|b\|_{\rm BMO}\|f\|_{{\mathcal M}^p_{L\log L}}.
\end{equation}
Gathering these estimates \eqref{eq:F1} and \eqref{eq:F2}, and the embedding ${\mathcal M}^p_1({\mathbb R}^n)\hookleftarrow{\mathcal M}^p_{L\log L}({\mathbb R}^n)$, we obtain
\begin{align*}
\||b,I_\alpha|f_2\|_{{\rm W}{\mathcal M}^s_t}
\lesssim
\|b\|_{\rm BMO}\|f\|_{{\mathcal M}^p_1}
+
\|b\|_{\rm BMO}\|f\|_{{\mathcal M}^p_{L\log L}}
\lesssim
\|b\|_{\rm BMO}\|f\|_{{\mathcal M}^p_{L\log L}},
\end{align*}
as desired.

\section{Proof of Theorem \ref{thm:main}}\label{s:proof}

In this section, we prove the main theorem.

At first, we start the proof of the \lq\lq only if'' part.
We take any $Q\in{\mathcal Q}({\mathbb R}^n)$.
By Proposition \ref{prop:normable-wLp} and Theorem \ref{thm:absolute}, $b\in{\rm BMO}({\mathbb R}^n)$ implies
\begin{align*}
\int_Q\int_{{\mathbb R}^n}\frac{|b(x)-b(y)|}{|x-y|^{n-\alpha}}|f(y)|\,{\rm d}y\,{\rm d}x
\lesssim
|Q|^{-\frac1s+\frac1t}\||b,I_\alpha|f\|_{{\rm W}{\mathcal M}^s_t}
\lesssim
|Q|^{-\frac1s+\frac1t}\|f\|_{{\mathcal M}^p_{L\log L}}
\end{align*}
for all $f\in{\mathcal M}^p_{L\log L}({\mathbb R}^n)$.
Then
\[
[b,I_\alpha]f(x)
=
\int_{{\mathbb R}^n}\frac{b(x)-b(y)}{|x-y|^{n-\alpha}}f(y)\,{\rm d}y
\]
converges a.e. $x\in{\mathbb R}^n$, absolutely.
In addition, using Theorem \ref{thm:absolute}, we obtain
\begin{align*}
\|[b,I_\alpha]f\|_{{\rm W}{\mathcal M}^s_t}
\le
\||b,I_\alpha|(|f|)\|_{{\rm W}{\mathcal M}^s_t}
\lesssim
\|f\|_{{\mathcal M}^p_{L\log L}}.
\end{align*}

Next, we show the \lq \lq if'' part.
We follow the idea of the paper \cite{Janson78} to prove the estimate
\begin{equation}\label{eq:if part}
\|b\|_{\rm BMO}
\lesssim
\|[b,I_\alpha]\|_{{\mathcal M}^p_{L\log L}\to{\rm W}{\mathcal M^s_t}}.
\end{equation}
Since the function $|z|^{n-\alpha}$ is infinitely differentiable in an open set, we can expand the absolutely convergent Fourier series
\[
|z|^{n-\alpha}
=
\sum_{j=1}^\infty a_je^{iv_j\cdot z}
\]
on $Q_0\in{\mathcal Q}({\mathbb R}^n)$ with $Q_0\not\ni0$ and $|Q_0|=1$ (cf. \cite[Corollary 3.3.10 (a)]{Grafakos14}).
In fact, we may construct a smooth $2Q_0$-periodic function such that 
$$
\chi_{Q_0}(z)|z|^{n-\alpha}
\le
\rho(z)
\le
\chi_{2Q_0}(z)|z|^{n-\alpha}
$$ 
for all $z\in 2Q_0$
and expand $\rho$
into
the absolutely convergent Fourier series on $2Q_0$, and restrict its expansion 
Here, the exact form of the vectors $\{v_j\}_{j=1}^\infty$ is irrelevant.
In particular, we remark that $\sum_{j=1}^\infty|a_j|<\infty$.
Here, for any $Q\in{\mathcal Q}({\mathbb R}^n)$, define
\[
R_Q
\equiv
Q-\ell(Q)Q_0
=
\left\{
x-\ell(Q)z
\,:\,
x\in Q,\,z\in Q_0
\right\}.
\]
Then, since $(x-y)/\ell(Q)\in Q_0$ for all $x\in Q$ and $y\in R_Q$, we can rewrite
\begin{align*}
\int_Q|b(x)-b_{R_Q}|\,{\rm d}x
&=
\int_Q(b(x)-b_{R_Q})s(x)\,{\rm d}x\\
&=
\frac1{|R_Q|}\int_Q\int_{R_Q}(b(x)-b(y))s(x)\,{\rm d}y\,{\rm d}x\\
&=
\frac1{|R_Q|}\int_Q\int_{R_Q}
(b(x)-b(y))
\left(\frac{|x-y|}{\ell(Q)}\right)^{-(n-\alpha)}
\sum_{j=1}^\infty a_je^{iv_j\cdot\frac{x-y}{\ell(Q)}}
s(x)
\,{\rm d}y\,{\rm d}x,
\end{align*}
where we wrote $s\equiv{\rm sgn}(b-m_{R_Q}(b))$.
Note that
\begin{align*}
\int_Q\int_{R_Q}
\frac{|b(x)-b(y)|}{|x-y|^{n-\alpha}}\sum_{j=1}^\infty|a_j|
\,{\rm d}y\,{\rm d}x
&=
\sum_{j=1}^\infty|a_j|
\int_Q|b,I_\alpha|\chi_{R_Q}(x)\,{\rm d}y\,{\rm d}x\\
&\lesssim
\sum_{j=1}^\infty|a_j|
\||b,I_\alpha|\chi_{R_Q}\|_{{\rm W}{\mathcal M}^s_t}\cdot|Q|^{1-\frac1s}\\
&\lesssim
\sum_{j=1}^\infty|a_j|
\|\chi_{R_Q}\|_{{\mathcal M}^p_{L\log L}}\cdot|Q|^{1-\frac1s}\\
&\sim
\sum_{j=1}^\infty|a_j|
\cdot
|Q|^{1+\frac\alpha n}
<\infty,
\end{align*}
where in the second inequality, we used Theorem \ref{thm:absolute}.
Hence, setting
\begin{equation*}
g_j(y)\equiv e^{-iv_j\cdot\frac y{\ell(Q)}}\chi_{R_Q}(y),
\quad
h_j(x)\equiv e^{iv_j\cdot\frac x{\ell(Q)}}s(x)\chi_Q(x),
\end{equation*}
by Fubini's theorem, we have
\begin{align*}
\int_Q|b(x)-b_{R_Q}|\,{\rm d}x
&\le
\ell(Q)^{-\alpha}
\sum_{j=1}^\infty|a_j|
\int_{{\mathbb R}^n}|[b,I_\alpha]g_j(x)||h_j(x)|\,{\rm d}x\\
&=
\ell(Q)^{-\alpha}
\sum_{j=1}^\infty|a_j|
\int_Q|[b,I_\alpha]g_j(x)|\,{\rm d}x\\
&\le
\ell(Q)^{-\alpha}
\sum_{j=1}^\infty|a_j|\cdot|Q|^{1-\frac1s}\|[b,I_\alpha]g_j\|_{{\mathcal M}^s_t}\\
&\lesssim
\ell(Q)^{-\alpha}
\|[b,I_\alpha]\|_{{\mathcal M}^p_{L\log L}\to{\rm W}{\mathcal M}^s_t}
\cdot|Q|^{1-\frac1s}
\sum_{j=1}^\infty|a_j|\|g_j\|_{{\mathcal M}^p_{L\log L}}\\
&\sim
\|[b,I_\alpha]\|_{{\mathcal M}^p_{L\log L}\to{\rm W}{\mathcal M}^s_t}
\cdot|Q|.
\end{align*}
Consequently, using the well-known fact
\[
\|b\|_{\rm BMO}
\sim
\sup_{Q\in{\mathcal Q}}\inf_{c\in{\mathbb C}}
\frac1{|Q|}\int_Q|b(x)-c|\,{\rm d}x,
\]
we obtain estimate \eqref{eq:if part}.

\appendix
\section{Elemental statements for the sharp maximal function and the median}

The definitions of the median $m_f(Q)$ and the sharp maximal function $M^{\sharp,d}_\lambda f(x)$ is mentioned in Subsection \ref{ss:sharp}.
Then, the following lemma is known for these notation.

\begin{lemma}\label{app:sharp}
For $f\in L_{\rm loc}^1({\mathbb R}^n)$ and $Q\in{\mathcal Q}({\mathbb R}^n)$, the following assrtions holds{\rm :}
\begin{itemize}
\item[{\rm (1)}] $(f\chi_Q)^\ast(\lambda|Q|)\le\lambda^{-1}|f|_Q$ for any $\lambda>0$.
\item[{\rm (2)}] \cite[Lemma 3.2]{Hytonen11} $|m_f(Q)|\le(f\chi_Q)^\ast(\lambda|Q|)$ for any $\lambda\in(0,2^{-1})$.
\item[{\rm (3)}] \cite[Proof of Proposition 3]{NaSa17} $M_\lambda^{\sharp,d}f(x)\lesssim M(|f|^\eta)(x)^{\frac1\eta}$ for any $\eta>0$.
\end{itemize}
\end{lemma}

The statement (1) in this lemma is given by Chebyshev's inequality, easily.

\section{Subadditivity of $\Phi$ and $\Psi$}

\begin{lemma}
Set
\[
\Phi(t)\equiv t\log(3+t), \quad
\Psi(t)\equiv\left(t\log\left(3+t^{\frac\alpha n}\right)\right)^{\frac n{n-\alpha}}.
\]
Then for all $s,t>0$,
\[
\Phi(st)\lesssim\Phi(s)\Phi(t), \quad
\Psi(st)\lesssim\Psi(s)\Psi(t).
\]
\end{lemma}

\begin{proof}
We calculate
\begin{align*}
\log(3+st)
\le
\log(3+s)(3+t)
=
\log(3+s)+\log(3+t)
\le
2\log(3+s)\log(3+t).
\end{align*}
Then
\begin{align*}
\Phi(st)
\le
st\cdot2\log(3+s)\log(3+t)
=
2\Phi(s)\Phi(t),
\end{align*}
and
\begin{align*}
\Psi(st)
\le
\left(st\log\left(3+s^{\frac\alpha n}t^{\frac\alpha n}\right)\right)^{\frac n{n-\alpha}}
\le
\left(
st
\cdot2
\log\left(3+s^{\frac\alpha n}\right)
\log\left(3+t^{\frac\alpha n}\right)
\right)^{\frac n{n-\alpha}}
=
2^{\frac n{n-\alpha}}\Psi(s)\Psi(t),
\end{align*}
as desired.
\end{proof}

\section{Proof of Lemma \ref{lem:elemental}}\label{App:elemental}

\begin{itemize}
\item[(1)] We prove the pointwise estimate
\[
\Psi\circ\Phi(t)
\lesssim
\left(t\log^2(3+t)\right)^{\frac n{n-\alpha}},
\quad t>0.
\]
Note that
\begin{align*}
\log\left(3+\Phi(t)^{\frac\alpha n}\right)
&\le
\log\left(3+[(3+t)\log(3+t)]^{\frac\alpha n}\right)
\le
\log\left(3+(3+t)^2\right)\\
&\le
3\log(3+t).
\end{align*}
Therefore,
\begin{align*}
\Psi\circ\Phi(t)
=
\left(\Phi(t)\log\left(3+\Phi(t)^{\frac\alpha n}\right)\right)^{\frac n{n-\alpha}}
\lesssim
\left(t\log^2(3+t)\right)^{\frac n{n-\alpha}}.
\end{align*}

\item[(2)] Note that
\begin{align*}
\Psi_0\left(t^2\right)
=
\left(\frac t{\log(3+1/t^2)}\right)^{\frac{2n}{n-\alpha}}
\ge
\left(\frac t{3+1/t^2}\right)^{\frac{2n}{n-\alpha}}
\gtrsim
\min\left(t,t^3\right)^{\frac{2n}{n-\alpha}},
\end{align*}
and then
\[
\Psi_0(t)
\gtrsim
\min(t,t^3)^{\frac n{n-\alpha}}.
\]
Then, taking as a continuous function
\[
\varphi_1(t)
\equiv
\begin{cases}
t^3, & t\le\dfrac1{\sqrt{3}}, \\
t-\dfrac2{3\sqrt{3}}, & t>\dfrac1{\sqrt{3}},
\end{cases}
\]
we have
\[
\Psi_0(t)
\gtrsim
\varphi_1(t)^{\frac n{n-\alpha}}.
\]
Consequently, we may define
\[
\Psi_1(t)
\equiv
\varphi_1(t)^{\frac n{n-\alpha}}
=
\begin{cases}
t^{\frac{3n}{n-\alpha}}, & t\le\dfrac1{\sqrt{3}}, \\
\left(t-\dfrac2{3\sqrt{3}}\right)^{\frac n{n-\alpha}}, & t>\dfrac1{\sqrt{3}},
\end{cases}
\]
and obtain $\Psi_1\in\nabla_2$.
\end{itemize}

\section{Proof of Theorem \ref{thm:sharp-pointwise}}\label{app:sharp-pointwise-proof}

In this appendix, we give the proof of Theorem \ref{thm:sharp-pointwise}.

Fix $Q\in{\mathcal Q}({\mathbb R}^n)$ and $x\in Q$, and decompose
\[
f=f\chi_{2Q}+f\chi_{{\mathbb R}^n\setminus2Q}\equiv f_1+f_2.
\]
Setting
\[
c_Q\equiv I_\alpha(|b-b_{2Q}|f_2)(x),
\quad
F_1(y)\equiv|b(y)-b_{2Q}|I_\alpha f(y),
\quad
F_2(y)\equiv I_\alpha(|b-b_{2Q}|f_1)(y),
\]
\[
F_3(y)\equiv|I_\alpha(|b-b_{2Q}|f_2)(y)-c_Q|
\]
for $y\in Q$, we can split
\[
||b,I_\alpha|f(y)-c_Q|
\le
F_1(y)+F_2(y)+F_3(y).
\]
Here, since $1/p-\alpha/n>0$, we claim that $c_Q$ is convergent.
In fact,
\begin{align*}
c_Q
&=
\int_{{\mathbb R}^n\setminus2Q}\frac{|b(z)-b_{2Q}|}{|x-z|^{n-\alpha}}f(z)\,{\rm d}z
\lesssim
\sum_{j=1}^\infty\frac1{(2^j\ell(Q))^{n-\alpha}}
\int_{2^{j+1}Q\setminus2^jQ}|b(z)-b_{2Q}|f(z)\,{\rm d}z\\
&\lesssim
\sum_{j=1}^\infty(2^j\ell(Q))^\alpha\|b-b_{2Q}\|_{\exp(L),2^jQ}\|f\|_{L\log L,2^jQ}
\lesssim
\sum_{j=1}^\infty|2^jQ|^{\frac\alpha n-\frac1p}
\|b\|_{\rm BMO}\|f\|_{{\mathcal M}^p_{L\log L}}\\
&\sim
|Q|^{\frac\alpha n-\frac1p}
\|b\|_{\rm BMO}\|f\|_{{\mathcal M}^p_{L\log L}}.
\end{align*}

At first, we estimate for $F_1(y)$.
By H\"older's inequality and Lemma \ref{lem:bmo} (1),
\begin{align*}
\frac1{|Q|}\int_Q|F_1(y)|\,{\rm d}y
&\le
\left(\frac1{|Q|}\int_Q|b(y)-b_{2Q}|^{\eta'}\,{\rm d}y\right)^{\frac1{\eta'}}
\left(\frac1{|Q|}\int_QI_\alpha f(y)^\eta\,{\rm d}y\right)^{\frac1\eta}\\
&\lesssim
\|b\|_{\rm BMO}M((I_\alpha f)^\eta)(x)^{\frac1\eta}.
\end{align*}

Next, we estimate for $F_2(y)$.
By Proposition \ref{prop:normable-wLp} and the weak-type $(1,n/(n-\alpha))$-boundedness of $I_\alpha$,
\begin{align*}
\frac1{|Q|}\int_Q|F_2(y)|\,{\rm d}y
&\lesssim
\frac1{|Q|^{\frac{n-\alpha}n}}\|I_\alpha(|b-b_{2Q}|f_1)\|_{{\rm W}L^{\frac n{n-\alpha}}}
\lesssim
\frac{\ell(Q)^\alpha}{|Q|}\||b-b_{2Q}|f_1\|_{L^1}\\
&\lesssim
\ell(Q)^\alpha\|b-b_{2Q}\|_{\exp(L),2Q}\|f\|_{L\log L,2Q},
\end{align*}
where in the third inequality we used the $L\log L$-$\exp(L)$ duality for a probability measure.
According to Lemma \ref{lem:bmo} (2), we obtain
\[
\frac1{|Q|}\int_Q|F_2(y)|\,{\rm d}y
\lesssim
\|b\|_{\rm BMO}M_{\alpha,L\log L}f(x).
\]

Finally, we estimate for $F_3(y)$.
Let $y\in Q$.
We estimate
\begin{align*}
|F_3(y)|
&\le
\int_{{\mathbb R}^n\setminus2Q}
\left|\frac1{|y-z|^{n-\alpha}}-\frac1{|x-z|^{n-\alpha}}\right||b(z)-b_{2Q}|f(z)
\,{\rm d}z\\
&\lesssim
\int_{{\mathbb R}^n\setminus2Q}
\frac{|x-y|}{|x-z|^{n-\alpha+1}}|b(z)-b_{2Q}|f(z)
\,{\rm d}z\\
&\lesssim
\sum_{j=1}^\infty\frac{\ell(Q)}{\ell(2^jQ)^{n-\alpha+1}}
\int_{2^{j+1}Q\setminus2^jQ}|b(z)-b_{2Q}|f(z)\,{\rm d}z\\
&\lesssim
\sum_{j=1}^\infty\frac{\ell(2^jQ)^\alpha}{2^{j(n+1)}}
\|b-b_{2Q}\|_{\exp(L),2^{j+1}Q}\|f\|_{L\log L,2^{j+1}Q},
\end{align*}
where in the last inequality we used the $L\log L$-$\exp(L)$ duality for a probability measure.
Using Lemma \ref{lem:bmo} (3), we have
\begin{align*}
|F_3(y)|
\lesssim
\sum_{j=1}^\infty\frac j{2^{j(n+1)}}
\|b\|_{\rm BMO}M_{\alpha,L\log L}f(x)
\lesssim
\|b\|_{\rm BMO}M_{\alpha,L\log L}f(x).
\end{align*}

Consequently, combining these estimates for $F_1(y)$, $F_2(y)$ and $F_3(y)$, we finish the proof of Theorem \ref{thm:sharp-pointwise}.

Finally, remark that as is mentioned in the proof of Theorem 1.3 in \cite{CrFi03}, Theorem \ref{thm:sharp-pointwise} can be improved as follows:

\begin{theorem}
Let $0<\alpha<n$ and $1<p<\infty$.
If
\[
b\in{\rm BMO}({\mathbb R}^n),
\quad
\frac1p-\frac\alpha n>0,
\]
then
\[
(|b,I_\alpha|f)^\sharp(x)
\lesssim
\|b\|_{\rm BMO}
\left(
I_\alpha f(x)+M_{\alpha,L\log L}f(x)
\right),
\quad x\in{\mathbb R}^n
\]
for all nonnegative functions $f\in{\mathcal M}^p_{L\log L}({\mathbb R}^n)$.
\end{theorem}

{\bf Acknowledgements.} 
The authors are thankful to Professor Yoshihiro Sawano, Professor Toru Nogayama and Dr.\ Kazuki Kobayashi for their careful reading the paper and giving very helpful comments and advice.
The author (N.H.) was supported by Grant-in-Aid for Research Activity Start-up Grant Number 23K19013.

\end{document}